\documentclass[11pt]{article}
\usepackage[a4paper]{anysize}\marginsize{3.5cm}{3.5cm}{1.3cm}{2cm}
\pdfpagewidth=\paperwidth \pdfpageheight=\paperheight
\usepackage{pgf,tikz,hyperref}
\usetikzlibrary{arrows}
\usepackage{amsfonts,amssymb,amsthm,amsmath,eucal}

\newcommand\numberthis{\addtocounter{equation}{1}\tag{\theequation}}

\theoremstyle{plain}
\newtheorem{thm}{Theorem}
\newtheorem{theorem}[thm]{Theorem}

\newtheorem{lemma}[thm]{Lemma}

\newtheorem{conjecture}[thm]{Conjecture}

\theoremstyle{definition}
\newtheorem{definition}[thm]{Definition}
\newtheorem{remark}[thm]{Remark}

\newtheorem{problem}[thm]{Problem}

\newtheorem{thevarthm}[thm]{\varthmname}

\newenvironment{varthm*}[1]{\trivlist\item[]{\bf #1.}\it}{\endtrivlist}

\newcommand\eps{\varepsilon}
\newcommand\be{\begin{eqnarray*}}
\newcommand\ee{\end{eqnarray*}}
\newcommand\set[1]{\left\{#1\right\}}
\newcommand\with{\,\,\vrule\,\,}
\newcommand\binomial[2]{{#1\choose#2}}

\newcommand\C{\mathbb C}
\newcommand\call{\mathcal L}
\newcommand\cali{\mathcal I}
\newcommand\calo{\mathcal O}
\newcommand\Z{\mathbb Z}
\renewcommand\P{\mathbb P}
\renewcommand\O{\mathcal O}
\newcommand\newop[2]{\def#1{\mathop{\rm #2}\nolimits}}
\newop\edim{edim}
\newop\Zeroes{Zeroes}
\newop\Jac{Jac}
\newcommand\rationalto{\dashrightarrow}
\newcommand\eqnref[1]{(\ref{#1})}

\newcommand\lnd{\call_N(d;m_1,\ldots,m_s)}
\newcommand\ltd{\call_3(d;m_1,\ldots,m_s)}
\newcommand\lndz{\call_N(d;Z,m_1,\ldots,m_s)}

\newcommand\keywords[1]{{\renewcommand\thefootnote{}\footnotetext{\textit{Keywords:} #1.}}}
\newcommand\subclass[1]{{\renewcommand\thefootnote{}\footnotetext{\textit{Mathematics Subject Classification (2010):} #1.}}}

\begin{document}

\author{Thomas~Bauer, Grzegorz~Malara, Tomasz~Szemberg, Justyna~Szpond}
\title{Quartic unexpected curves and surfaces}
\date{\today}
\maketitle
\thispagestyle{empty}

\begin{abstract}
   Our research is motivated by recent work of
   Cook II, Harbourne, Migliore, and Nagel
   on configurations of points in
   the projective plane with properties that are unexpected from
   the point of view of the postulation theory. In this note, we
   revisit the basic configuration of nine points
   appearing in
   work of
   Di Gennaro/Ilardi/Vall\`es
   and Harbourne,
   and we exhibit some additional new properties of this configuration.
   We then pass to projective three-space $\P^3$ and exhibit
   a surface with unexpected postulation properties there.
   Such higher dimensional phenomena have not been observed so far.
\end{abstract}

\subclass{MSC 14C20 \and MSC 14J26 \and MSC 14N20 \and MSC 13A15 \and MSC 13F20}


\section{Introduction}
   Let $P_1,\dots,P_s$ be a set of $s\geq 1$ generic points in $\P^N(\C)$ and
   let $m_1,\ldots,m_s$ be positive integers. It is a classical problem
   in algebraic geometry to study the linear systems $\lnd$ of hypersurfaces of degree $d$
   passing through each of the points $P_i$ with multiplicity at least $m_i$ for $i=1,\ldots,s$.
   This linear system is viewed as the projectivization of the vector space of
   homogeneous polynomials of degree $d$ vanishing at points $P_1,\ldots,P_s$
   to order $m_1,\ldots,m_s$ respectively. Determining its projective dimension is
   one of the fundamental questions in the area.

\begin{problem}
   Determine $\dim\lnd$.
\end{problem}

   The \emph{expected dimension} of $\lnd$ is the number given by the naive conditions count
   \begin{equation}\label{eq:edim}
      \edim\lnd=\max\left\{-1,\binom{N+d}{N}-\sum_{i=1}^s\binom{N+m_i-1}{N}-1\right\}
      \,.
   \end{equation}
   We have always
   \begin{equation}\label{eq:dim edim}
      \dim\lnd\geq \edim\lnd
      \,.
   \end{equation}
   If equality holds in \eqref{eq:dim edim}, then we say that the system $\lnd$
   is \emph{non-special}. Otherwise it is called \emph{special}.

\begin{remark}\label{rem:nonspecial}
   It is well known that a single point with arbitrary multiplicity (that is, $s=1$ and $m_1$
   arbitrary) is non-special for any $d$ and $N$. Similarly, generic points with multiplicity $1$
   (that is $m_1=\ldots=m_s=1$) impose always independent conditions on forms of arbitrary
   degree in projective spaces of arbitrary dimension.
\end{remark}

   Special linear systems with multiplicities $m_1=\ldots=m_s=2$ have been completely
   classified by Alexander and Hirschowitz \cite{AleHir95}.
   The Segre-Harbourne-Gimigliano-Hirschowitz Conjecture governs the speciality
   of planar systems with points of arbitrary multiplicity, see \cite{Cil00} for
   a very nice survey. In $\P^3$ the special linear systems $\ltd$ are the subject
   of a conjecture due to Laface and Ugaglia \cite{LafUga06}. In higher dimensions
   there are some partial results due to Alexander and Hirschowitz \cite{AleHir00}
   and scattered partial conjectures due to various authors, see e.g.~\cite{EmsIar95}.
   The complete picture remains however rather obscure.

   In the groundbreaking article \cite{CHMN}, Cook II, Harbourne, Migliore and Nagel
   opened a new path of research. They propose to study systems
   \begin{equation}\label{eq:type}
      \lndz
      \,,
   \end{equation}
   where $Z$ is a finite set of points (with multiplicity $1$)
   and $P_1,\ldots,P_s$ are generic fat points, i.e., $m_1,\ldots,m_s\geq 2$.
   Thus the classical problem outlined above is the case $Z=\emptyset$.
   In \cite{CHMN} the authors focus on the case $N=2$ and $s=1$. They show that,
   somewhat unexpectedly in the view of Remark \ref{rem:nonspecial},
   there exist special linear systems in this situation. They relate
   the existence of such systems to properties of line arrangements
   determined by lines dual to points in $Z$. This leads to a very nice geometric
   explanation of the existence of special curves.

   In the present note we exhibit a new phenomenon:
   The existence of special linear systems of type \eqref{eq:type} in a higher
   dimensional projective space, namely in $\P^3$. In the subsequent paper
   \cite{Szp18}, the last named author will show how our example can be generalized
   to higher dimensional projective spaces.

   Our main result is Theorem \ref{thm:unexpected quartic}.
   Conjecture \ref{conj:syz} proposes a geometric explanation for the existence
   of the special surfaces in Theorem \ref{thm:unexpected quartic}. Our research has been accompanied
   by Singular \cite{DGPS} experiments.


\section{Plane quartics with nine base points and a general triple point}\label{sec:B3}

   We begin with a Coxeter arrangement of lines classically denoted by B3 (or $A(9,1)$ in Gr\"unbaum's notation).
   This arrangement is given by the linear factors of the polynomial
   \be
      f=xyz(x+y)(x-y)(x+z)(x-z)(y+z)(y-z)
      \,.
   \ee
   It is depicted in Figure \ref{fig:9-lines}, with the convention that the line at infinity
   $z=0$ is indicated by the circle.
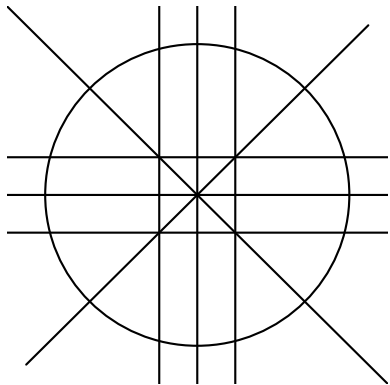
\begin{figure}
\begin{center}
\definecolor{qqttzz}{rgb}{0,0,0}
\definecolor{fqrutc}{rgb}{0,0,0}
\definecolor{qqwuqq}{rgb}{0,0,0}
\begin{tikzpicture}[line cap=round,line join=round,>=triangle 45,x=1.0cm,y=1.0cm,scale=0.5]
\clip(-5,-5) rectangle (5,5);
\draw [line width=.8pt,color=qqwuqq] (0.,-5.54) -- (0.,5.8);
\draw [line width=.8pt,color=qqwuqq,domain=-7.82:13.18] plot(\x,{(1.-0.*\x)/1.});
\draw [line width=.8pt,color=qqwuqq,domain=-7.82:13.18] plot(\x,{(-0.-0.*\x)/1.});
\draw [line width=.8pt,color=qqwuqq,domain=-7.82:13.18] plot(\x,{(-1.-0.*\x)/1.});
\draw [line width=.8pt,color=qqwuqq] (-1.,-5.54) -- (-1.,5.8);
\draw [line width=.8pt,color=qqwuqq] (1.,-5.54) -- (1.,5.8);
\draw [line width=.8pt,color=qqwuqq,domain=-7.82:13.18] plot(\x,{(-0.-1.*\x)/1.});
\draw [line width=.8pt,color=qqwuqq,domain=-4.5:4.5] plot(\x,{(-0.-1.*\x)/-1.}); 
\draw [line width=.8pt,color=fqrutc] (0.,0.) circle (4.cm);
\end{tikzpicture}
\end{center}
\caption{The $B3$ arrangement of lines.}
\label{fig:9-lines}
\end{figure}
   We use the following convention for
   the duality between lines and points: A point $(A:B:C)$
   corresponds to the line $Ax+By+Cz=0$ and vice versa.
   Thus the linear factors of $f$ correspond to the points
   $$\begin{array}{lll}
   P_1=(1:0:0),& P_2=(0:1:0),& P_3=(0:0:1),\\
   P_4=(1:1:0),& P_5=(1:-1:0),& P_6=(1:0:1),\\
   P_7=(1:0:-1),& P_8=(0:1:1),& P_9=(0:1:-1).
   \end{array}$$
   It was observed in \cite{CHMN} that these points
   impose independent conditions on quartics in $\P^2$.

\begin{remark}
   There is some ambiguity in the choice of coordinates of the points $P_1,\ldots,P_9$.
   Harbourne introduced in \cite[Example 4.1.10]{Harbourne:asymptotics} some coordinates
   which due to their lack of symmetry are not as convenient to work with as those used
   by Dimca in \cite[Example 3.6]{Dim16}. Therefore we prefer to work with Dimca's coordinates.
\end{remark}

   Let $V\subset H^0(\P^2,\O_{\P^2}(4))$ be the linear subspace
   of all quartics vanishing at $P_1,\dots,P_9$.
   As we have $\dim V=6$ by \cite{CHMN},
   one expects that the system becomes empty when
   $6$ further conditions are imposed.
   Unexpectedly, however, there exists
   for any choice of an additional point $R=(a:b:c)$
   a quartic $Q_R\in V$ with a triple point at~$R$.
   This quartic can be written down explicitly as
   \begin{align*}
      Q_R(x:y:z) ={} & 3a(b^2-c^2)\cdot x^2yz +3b(c^2-a^2)\cdot xy^2z +3c(a^2-b^2)\cdot xyz^2\\
                &+a^3\cdot y^3z -a^3\cdot yz^3 +b^3\cdot xz^3
                -b^3\cdot x^3z +c^3\cdot x^3y -c^3\cdot xy^3. \numberthis\label{eq:quartic}
   \end{align*}
   It is elementary to check that $Q_R$ indeed vanishes
   at the points $P_1,\dots,P_9$ and vanishes at $R$ to order 3.

   The equation \eqref{eq:quartic} can be viewed also
   as a cubic equation in the variables $a,b,c$ with parameter $S=(x:y:z)$.
   Taking this dual point of view, it becomes
   \begin{align*}
      Q_S(a:b:c)={}&yz(y^2-z^2)\cdot a^3 +xz(z^2-x^2)\cdot b^3 +xy(x^2-y^2)\cdot c^3 +3x^2yz\cdot ab^2  \\
                &-3xy^2z\cdot a^2b +3xyz^2\cdot a^2c -3x^2yz\cdot ac^2 +3xy^2z\cdot bc^2- 3xyz^2\cdot b^2c
                \,.
                \numberthis\label{eq:quartic-dual}
   \end{align*}
   (Of course the equations \eqnref{eq:quartic} and \eqnref{eq:quartic-dual}
   are the same, just treated from a different perspective.)
   Surprisingly, each of the cubic curves \eqnref{eq:quartic-dual}
   has a triple point at $S=(x:y:z)$.
   All the cubics therefore split in a product of three lines --
   except for those with
   parameter values $(x:y:z)$ that correspond to the points $P_1,\dots,P_9$,
   in which cases the right hand side of equation \eqref{eq:quartic-dual} vanishes identically.
   Currently we do not have a theoretical explanation for this property.

   The family of cubics $Q_{S}$ parameterized by
   $S=(x:y:z)\in\P^2$ has no additional base points. This can be easily
   verified for specific and sufficiently general values
   of $(x:y:z)$.
\begin{remark}
   Recently, Farnik, Galuppi, Sodomaco and Trok have announced an interesting result
   to the effect that the quartic curve discussed in this section is, up to projective equivalence,
   the only unexpected curve of this degree, see \cite{FGST}.
\end{remark}

\section{Quartic surfaces with 31 base points and a general triple point}\label{sec:Fermat}

   Let $F$ be the Fermat-type ideal in $\C[x,y,z,w]$ generated by
   \be
      x^3-y^3,\
      y^3-z^3,\
      z^3-w^3.
   \ee
   Its zero locus $Z$ consists of the $27$ points
   \be
      P_{(\alpha,\beta,\gamma)}=(1:\eps^\alpha:\eps^\beta:\eps^\gamma)
   \ee
   where $\eps$ is a primitive root of unity of order 3 and
   $1\le\alpha,\beta,\gamma\le 3$.
   Let $I$ be the ideal of the union $W$ of $Z$ with the $4$ coordinate points,
   that is,
   $$I=F\cap (x,y,z)\cap (x,y,w)\cap (x,z,w)\cap (y,z,w).$$

\begin{lemma}\label{lem:generators of I}
   The ideal $I$ is generated by the following $8$ binomials of degree $4$:
   $$x(y^3-z^3),\; x(z^3-w^3),\; y(x^3-z^3),\; y(z^3-w^3)\,,$$
   $$z(x^3-y^3),\; z(y^3-w^3),\; w(x^3-y^3),\; w(y^3-z^3)\,.$$
\end{lemma}

\begin{proof}
   Let $J$ be the ideal generated by the $8$ binomials. We show first
   the containment $J\subset I$. Since $I$ is a radical ideal
   by definition, it is enough to check that $\Zeroes(I)\subset\Zeroes(J)$.
   To this end it is enough to verify that every binomial generating $J$ vanishes along $W$.
   This is obvious because the vanishing in the coordinate points is guaranteed
   by the fact that every binomial involves $3$ different variables and
   vanishing along $Z$ is provided by the cubic term in brackets.

   For the reverse containment let $f$ be an element of $I$. In particular, $f$
   vanishes at all points of $Z$, so we may write $f$ in the following way
   \begin{equation}\label{eq:f}
   f=g_y\cdot(x^3-y^3)+g_z\cdot(x^3-z^3)+g_w\cdot(x^3-w^3)
   \,,
   \end{equation}
   with homogeneous polynomials $g_y,g_z,g_w$.
   From now on we work modulo $J$. We want to show that $f=0$.
   Since $z(x^3-y^3), w(x^3-y^3)\in J$, we may assume that $g_y$ depends only on $x$ and $y$.
   By the same token, we may assume that $g_z$ depends only on $x$ and $z$, and $g_w$ respectively
   on $x$ and $w$.

   We have
   $$xy(x^3-y^3)=-y\cdot x(y^3-z^3)+x\cdot y(x^3-z^3)\,,$$
   so that $xy(x^3-y^3)\in J$. Thus $g_y=a_yx^d+b_yy^d$ for some $a_y,b_y\in \C$ and $d\geq 0$.
   Similarly $g_z=a_zx^d+b_zz^d$ and $g_w=a_wx^d+b_ww^d$.

   Evaluating \eqref{eq:f} at $(0:1:0:0)$ we obtain
   $$0=f(0:1:0:0)=g_y(0:1:0:0)=b_y.$$
   Similarly, $b_z=b_w=0$.
   Thus
   \begin{equation}\label{eq:f z a}
      f=x^d\left(a_y(x^3-y^3)+a_z(x^3-z^3)+a_w(x^3-w^3) \right)
      \,.
   \end{equation}
   If $d=0$, then evaluating again in the coordinate points $(0:1:0:0)$, $(0:0:1:0)$ and $(0:0:0:1)$
   we obtain $a_y=a_z=a_w=0$ and we are done.

   If $d>0$, then evaluating at $(1:0:0:0)$ we get from \eqref{eq:f z a}
   \begin{equation}\label{eq:as}
      a_y+a_z+a_w=0\,.
   \end{equation}
   Since $xy^3=xw^3$ and $xz^3=xw^3$ modulo $J$, we get from \eqref{eq:f z a} and \eqref{eq:as}
   $$
      f=x^{d-1}\left(a_yx(x^3-w^3)+a_zx(x^3-w^3)+a_wx(x^3-w^3) \right)=0
   $$
   and we are done.
\end{proof}

   Let $V\subset H^0(\P^3,\O_{\P^3}(4))$ be the linear space
   of quartics vanishing along $W$. It follows from Lemma \ref{lem:generators of I}
   that $\dim(V)=8$. As vanishing to order $3$ at a point in $\P^3$
   imposes $10$ conditions on forms of arbitrary degree, we do not expect
   that for a general (hence for any) point $R=(a:b:c:d)$ there exists
   a quartic $Q_R\in V$
   vanishing to order three at $R$. However this is the case,
   as we now show:

\begin{theorem}\label{thm:unexpected quartic}
   The system
   $\call_3(4;W,3)$ is special, i.e.,
   for any point
   $R=(a:b:c:d)$ in $\P^3\setminus W$ there
   exists a quartic $Q_R$ vanishing to order $3$ at $R$ and vanishing
   at all points from the set $W$. Moreover,
   the quartic $Q_R$ has 4 additional singularities at
   \be
      R_1 = (-2a:b:c:d)\,,\
      R_2 = (a:-2b:c:d)\,, \\
      R_3 = (a:b:-2c:d)\,,\
      R_4 = (a:b:c:-2d)\,.
   \ee
   These singularities are double points.
\end{theorem}

\begin{proof}
   Since we know from Lemma \ref{lem:generators of I} that $V$ is
   of dimension $8$, it is
   enough to prove the existence of a quartic as claimed.
   It can be checked by elementary calculations that the quartic
   \begin{align*}
      Q_R(x:y:z:w)&=b^2(c^3-d^3)\cdot x^3y+a^2(d^3-c^3)\cdot xy^3+c^2(d^3-b^3)\cdot x^3z\\
                  &+c^2(a^3-d^3)\cdot y^3z+a^2(b^3-d^3)\cdot xz^3+b^2(d^3-a^3)\cdot yz^3\\
                  &+d^2(b^3-c^3)\cdot x^3w+d^2(c^3-a^3)\cdot y^3w+d^2(a^3-b^3)\cdot z^3w\\
                  &+a^2(c^3-b^3)\cdot xw^3+b^2(a^3-c^3)\cdot yw^3+c^2(b^3-a^3)\cdot zw^3 \numberthis\label{eq:quartic in P3}
   \end{align*}
   satisfies the assertion.
\end{proof}

   As before, the equation in \eqref{eq:quartic in P3} can be viewed
   as a quintic polynomial $Q_S$ in variables $a,b,s,d$. Let $S=(x:y:z:w)$, then we have
   \begin{align*}
      Q_S(a:b:c:d)&=y(w^3-z^3)\cdot a^3b^2+x(z^3-w^3)\cdot a^2b^3+z(y^3-w^3)\cdot a^3c^2\\
                  &+z(w^3-x^3)\cdot b^3c^2+x(w^3-y^3)\cdot a^2c^3+y(x^3-w^3)\cdot b^2c^3\\
                  &+w(z^3-y^3)\cdot a^3d^2+w(x^3-z^3)\cdot b^3d^2+w(y^3-x^3)\cdot c^3d^2\\
                  &+x(y^3-z^3)\cdot a^2d^3+y(z^3-x^3)\cdot b^2d^3+z(x^3-y^3)\cdot c^2d^3.
   \end{align*}
   It can be checked by elementary computation that $Q_S$ has a triple point at $S$.


\section{Geometry of the unexpected quartic}

   It is well known that a quartic surface $X$ in $\P^3$ with a triple point
   is rational. (This is easily seen by projecting $X$ from the triple point.)
   More importantly, $X$ is the image of $\P^2$ under the rational
   mapping $\varphi$ defined by the linear system of plane quartics vanishing
   along a complete intersection $0$-dimensional subscheme $U$ of length $12$,
   see~\cite{Roh84}. In this section, for $X=Q_R$, we identify the subscheme $U$ and the mappings
   explicitly.

   We begin with the projection.
   Let $\pi:\P^3\dashrightarrow\P^2$ be the projection from $R=(a:b:c:d)$ onto
   the plane $\P^2$ with coordinates $(p:q:r)$, which is defined by
   $$\pi: (x:y:z:w) \mapsto (p:q:r)=(dz-cw:cy-bz:bx-ay).$$
   Let $L_i$ be the line joining $R$ and $R_i$ for
   $i=1,\ldots,4$ (where the $R_i$ are the points from
   Theorem~\ref{thm:unexpected quartic}).
   By Bezout's theorem,
   these lines are contained in $Q_R$. They get contracted under $\pi$ onto
   points $F_i$ respectively, where
   $$F_1=(0:0:1),\; F_2=(0:-c:a),\; F_3=(-d:b:0),\; F_4=(1:0:0).$$
   Taking some generic sections of $Q_R$ and their images in $\P^2$ we determine
   $4$ points
   $$B_1=(dc:-cb:ab),\; B_2=(dc:0:-ab),\; $$
   $$B_3=(a^2b^2(c^3-d^3):a^2d^2(b^3-c^3):c^2d^2(a^3-b^3)),\; B_4=(0:1:0).$$
   which will be additional assigned base points of the linear series
   of quartics that we will consider.
   Let $\Gamma_R$ be the plane cubic given by the equation
   \begin{align*}
   \Gamma_R(p:q:r)&=bc^2d(a^3-b^3)\cdot p^2q+c^2d^2(a^3-b^3)\cdot pq^2\\
                  &+a^2bd(c^3-b^3)\cdot p^2r+2a^2d^2(c^3-b^3)\cdot pqr\\
                  &+a^2b^2(c^3-d^3)\cdot q^2r+acd^2(c^3-b^3)\cdot pr^2+ab^2c(c^3-d^3)\cdot qr^2
   \end{align*}
   and let $\Delta$ be the plane quartic defined by the equation
   \begin{align*}
   \Delta_R(p:q:r)&=a^2bc^2(a^3-b^3)(c^3-d^3)\cdot pq^2r +a^2c^2d(a^3-b^3)(c^3-d^3)\cdot q^3r\\
                  &-ab^2d^2(a^3-c^3)(b^3-c^3)\cdot p^2r^2+b^3cd(a^3-c^3)(c^3-d^3)\cdot qr^3 \\
                  &+ad(c^3-d^3)(a^3b^3+a^3c^3-2b^3c^3)\cdot q^2r^2 -bcd^3(a^3-c^3)(b^3-c^3)\cdot pr^3\\
                  &-ab(b^3c^6-a^3c^6+2a^3b^3d^3-a^3c^3d^3-3b^3c^3d^3+2c^6d^3)\cdot pqr^2.
   \end{align*}
   Let $U_R$ be the scheme-theoretic complete intersection of $\Gamma_R$ and $\Delta_R$. Then $U$
   has length $12$ and it is supported on the $8$ points $B_1,\ldots, B_4$ and $F_1,\ldots,F_4$.
   The points $B_i$ for $i=1,\ldots,4$ are reduced in $U$.
   The points $F_i$ for $i=1,\ldots,4$ support each a structure of length $2$.
   In these points the curves $\Gamma_R$ and $\Delta_R$ are tangent to each other.

   The linear system $|\calo_{\P^2}(4)\otimes\cali_U|$ has (projective) dimension $3$.
   It is spanned by $f_0=p\Gamma_R$, $f_1=q\Gamma_R$, $f_2=r\Gamma_R$, and $f_3=\Delta_R$. In order to recover the original
   coordinates of all points, it is however necessary to define
   $\varphi=(g_0:g_1:g_2:g_3):\P^2\dashrightarrow\P^3$ in the following basis:
   \begin{align*}
      g_0&=abc^2(a^3-b^3)\; f_0+2ac^2d(a^3-b^3)\; f_1
         +d(a^3b^3+2a^3c^3-3b^3c^3)\; f_2-ab^2\; f_3,\\
      g_1&=b^2c^2(a^3-b^3)\; f_0+2bc^2d(a^3-b^3)\; f_1
         +a^2bd(b^3-c^3)\; f_2-b^3\; f_3,\\
      g_2&=bc^3(a^3-b^3)\; f_0-c3d(a^3-b^3)\; f_1
         +a^2cd(b^3-c^3)\; f_2-b^2c\; f_3,\\
      g_3&=-2bc^2d(a^3-b^3)\; f_0-c^2d^2\; f_1+a^2d^2(b^3-c^3)\; f_2-b^2d\; f_3.
   \end{align*}
   Note that the mapping $\varphi$ contracts the cubic curve $\Gamma$ to the
   triple point $R$.

   All claims in this section can be in principle checked by tedious
   hand calculations. In order to allow a more convenient verification, we provide
   a Singular code in \cite{Sin18}.


\section{General geometric considerations}

   Cook et al.\ establish in \cite[Prop.~5.10]{CHMN} a method that allows
   to determine unexpected curves from syzygies.
   We propose here a conjecture that generalizes their idea to the
   surface case.
   To set it up, we need to generalize the notions
   of \emph{multiplicity index} and \emph{speciality index}
   that were introduced in
   \cite{CHMN} for the case of $\P^2$.

\begin{definition}\rm
   Let $Z$ be a reduced
   0-dimensional subscheme of $\P^n$.
   The \emph{multiplicity index} of $Z$ is the number
   \be
      m_Z = \min\set{j\in\Z\with\dim [I_{Z+jP}]_{j+1}>0}
   \ee
   where $P$ is a general point in $\P^n$.

   The \emph{speciality index} $u_Z$ of $Z$ is the least integer $j$
   such that, for a general point $P\in\P^n$, the scheme
   $Z+jP$ imposes independent conditions
   on the system $|\O_{\P^n}(j+1)|$, i.e., the smallest $j$
   such that
   \be
      \dim [I_{Z+jP}]_{j+1} = \binomial{j+1+n}{n} - \binomial{n-1+j}{n} - |Z|
      \,.
   \ee
\end{definition}

\medskip\noindent
   Consider now a reduced scheme $Z\subset\P^3$ of $d$ points $P_i$.
   For each point $P_i$ let $\ell_i\in K[x,y,z,w]$ be a linear form
   defining the plane dual to $P_i$, and set $f=\ell_1\cdot\ldots\cdot\ell_d$.
   Further, let $\ell$ be a linear form
   defining a general plane in $\P^3$.

\begin{conjecture}\label{conj:syz}
   Assume that the characteristic of $K$ does not divide $|Z|$
   and that
   with $m_Z\le u_Z$.
   Let
   \be
      s_0 f_x + s_1 f_y + s_2 f_z + s_3 f_w + s_4 \ell & = & 0 \\
      s'_0 f_x + s'_1 f_y + s'_2 f_z + s'_3 f_w + s'_4 \ell & = & 0
   \ee
   be linearly independent syzygies
   of least degree
   of the ideal
   $\Jac(f)+(\ell)=(f_x,f_y,f_z,f_w,\ell)$, and consider the rational maps
   $\P^3\rationalto\P^3$ given as
   \be
      \sigma = (s_0:s_1:s_2:s_3)
      \quad\mbox{and}\quad
      \sigma' = (s'_0:s'_1:s'_2:s'_3)
      \,.
   \ee
   Further,
   consider the rational map
   \be
      \Phi: \P^3 & \rationalto & (\P^3)^* \\
      Q          & \mapsto     & \mbox{the plane through $Q$, $\sigma(Q)$, $\sigma'(Q)$}
   \ee
   Then the image of the restriction of $\Phi$ to $\ell$
   is an unexpected surface for $Z$.
\end{conjecture}

\begin{remark}\rm
   The conjecture above is supported by the following
   considerations:
   \begin{itemize}
   \item[(1)]
      \emph{For each $i$,
      all points on the line $\ell\cap\ell_i$
      are mapped to the point $P_i$
      (i.e., the line goes through $Z$, as desired).}

      \medskip
      \emph{Proof.}
      Let $Q$ be a point on $\ell\cap\ell_i$.
      One shows first
      that $\sigma(Q)\in\ell_i$
      (this works as in \cite[Prop.~5.10]{CHMN}).
      But we also have
      $\sigma'(Q)\in\ell_i$ by the same argument.
      Thus, the three points
      $Q$, $\sigma(Q)$, $\sigma'(Q)$ all lie on $\ell_i$.
      By definition of $\Phi$, this implies
      $\Phi(Q)=\ell_i$.
   \item[(2)]
      \emph{The points in $\ell\cap\sigma(\ell)\cap\sigma'(\ell)$
      map to the general point $P$.}

      \medskip
      \emph{Proof.}
      Let $Q$ be a point in $\ell\cap\sigma(\ell)\cap\sigma'(\ell)$.
      As above, we have then $\sigma(Q)\in\ell$ and $\sigma'(Q)\in\ell$.
      By definition of $\Phi$ it follows that
      $\Phi(Q)=\ell$.
   \end{itemize}
   To prove the conjecture, one would need
   to show:
   \begin{itemize}
   \item[(a)]
      The image of $\Phi$ is a surface.
   \item[(b)]
      $\Phi$ is undefined only in certain points
      of the lines $\ell\cap\ell_i$.
      (This would follow from the following condition:
      For every point $Q$ on $\ell$, the points
      $\sigma(Q)$ and $\sigma'(Q)$ are not collinear.
      In other words, the syzygy vectors
      $(s_0,s_1,s_2)$ and $(s'_0,s'_2,s'_2$) are linearly
      independent in all points of $\ell$.)
   \item[(c)]
      The multiplicity of the point $P$ in (2) is high enough.
   \end{itemize}

\end{remark}


\paragraph*{Acknowledgement.}
   This research has been initiated while the three last authors visited the University
   of Marburg. It is a pleasure to thank the Department of Mathematics in Marburg for hospitality and
   Istv\'an Heckenberger and Volkmar Welker for helpful conversations.
   Malara was partially supported by National Science Centre, Poland, grant 2016/21/N/ST1/01491.
   Szemberg and Szpond were partially supported by National Science Centre, Poland, grant
   2014/15/B/ST1/02197.



\bigskip
\bigskip
\small

   Thomas Bauer,
   Fach\-be\-reich Ma\-the\-ma\-tik und In\-for\-ma\-tik,
   Philipps-Uni\-ver\-si\-t\"at Mar\-burg,
   Hans-Meer\-wein-Stra{\ss}e,
   D-35032~Mar\-burg, Germany

\nopagebreak
   \textit{E-mail address:} \texttt{tbauer@mathematik.uni-marburg.de}

\bigskip
   Grzegorz Malara,
   Department of Mathematics, Pedagogical University of Cracow,
   Podchor\c a\.zych 2,
   PL-30-084 Krak\'ow, Poland.

\nopagebreak
   \textit{E-mail address:} \texttt{grzegorzmalara@gmail.com}

\bigskip
   Tomasz Szemberg,
   Department of Mathematics, Pedagogical University of Cracow,
   Podchor\c a\.zych 2,
   PL-30-084 Krak\'ow, Poland

\nopagebreak
   \textit{E-mail address:} \texttt{tomasz.szemberg@gmail.com}

\bigskip
   Justyna Szpond,
   Department of Mathematics, Pedagogical University of Cracow,
   Podchor\c a\.zych 2,
   PL-30-084 Krak\'ow, Poland.

\nopagebreak
   \textit{E-mail address:} \texttt{szpond@up.krakow.pl}



\begin{thebibliography}{99}\footnotesize\itemsep=0cm

\bibitem{AleHir95}
   Alexander, J., Hirschowitz, A.:
   Polynomial interpolation in several variables.
   J. Algebraic Geom. 4 (1995), no. 2, 201--222

\bibitem{AleHir00}
   Alexander, J., Hirschowitz, A.:
   An asymptotic vanishing theorem for generic unions of multiple points.
   Invent. Math. 140 (2000), no. 2, 303--325

\bibitem{Cil00}
   Ciliberto, C.:
   Geometric aspects of polynomial interpolation in more variables and of Waring's problem.
   European Congress of Mathematics, Vol. I (Barcelona, 2000), 289--316, Progr. Math., 201, Birkhäuser, Basel, 2001.

\bibitem{CHMN}
   Cook II, D., Harbourne, B., Migliore, J., Nagel, U.:
   Line arrangements and configurations of points with an unusual geometric property.
   arXiv:1602.02300

\bibitem{DGPS}
   Decker, W.; Greuel, G.-M.; Pfister, G.; Sch{\"o}nemann, H.:
   \newblock {\sc Singular} {4-1-0} --- {A} computer algebra system for polynomial computations.
   \newblock {http://www.singular.uni-kl.de} (2016)

\bibitem{Gennaro-et-al}
   Di Gennaro, R., Ilardi, G., Vall\`es, J.:
   Singular hypersurfaces characterizing the Lefschetz properties.
   J. London Math. Soc. (2) 89 (2014), no. 1, 194--212

\bibitem{Dim16}
   Dimca, A.:
   Curve arrangements, pencils, and Jacobian syzygies.
   Michigan Math. J. 66 (2017), 347--365

\bibitem{EmsIar95}
   Emsalem, J., Iarrobino, A.:
   Inverse system of a symbolic power. I.
   J. Algebra 174 (1995), no. 3, 1080--1090

\bibitem{FGST}
   Farnik, {\L }., Galuppi, F., Sodomaco, L., Trok, W.:
   On the unique unexpected quartic in $\P^2$.
   preprint 2018

\bibitem{Harbourne:asymptotics}
   Harbourne, B.:
   Asymptotics of linear systems, with connections to line arrangements.
   arXiv:1705.09946 [math.AG]

\bibitem{LafUga06}
   Laface, A., Ugaglia, L.:
   On a class of special linear systems of $\P^3$.
   Trans. Amer. Math. Soc. 358 (2006), no. 12, 5485--5500

\bibitem{Roh84}
   Rohn, K.:
   Ueber die Fl\"achen vierter Ordnung mit dreifachem Punkte.
   Math. Ann. 24 (1884), no. 1, 55--151

\bibitem{Sin18}
   Singular script available at
   \verb|http://szpond.up.krakow.pl/quartic|

\bibitem{Szp18}
   Szpond, J.:
   Fermat-type configurations of points and hypersurfaces in $\P^N$ with unexpected
   geometric properties.
   In preparation.

\end{thebibliography}
\end{document}